\documentclass[12pt]{amsart}

\usepackage{amssymb}

\newtheorem{theorem}{Theorem}[section]
\newtheorem{proposition}[theorem]{Proposition}
\newtheorem{lemma}[theorem]{Lemma}
\newtheorem{remark}[theorem]{Remark}
\newtheorem{corollary}[theorem]{Corollary}

\newcommand{\calX}{{\mathcal X}}
\newcommand{\calM}{{\mathcal M}}
\newcommand{\calV}{{\mathcal V}}
\newcommand{\CP}{{\mathbb C}{\mathbb P}}
\newcommand{\Mlambda}{\calM_{\mathrm{Hod}}}
\newcommand{\MDH}{\calM_{\mathrm{DH}}}
\newcommand{\Mbun}{\calM_{r, {\mathcal O}_X}}
\newcommand{\Mbunhat}{\calM_{r, {\mathcal O}_X(x_0)}}
\newcommand{\Mbunbar}{\calM_{r, {\mathcal O}_{\overline{X}}}}
\newcommand{\Mconn}{\calM_{\mathrm{conn}}}
\newcommand{\Mrep}{\calM_{\mathrm{rep}}}
\newcommand{\MHiggs}{\calM_{\mathrm{Higgs}}}
\newcommand{\Treal}{T_{\mathbb{R}}}
\newcommand{\stab}{\mathrm{s}}
\newcommand{\smooth}{\mathrm{sm}}
\newcommand{\pr}{\mathrm{pr}}

\numberwithin{equation}{section}

\begin{document}
\baselineskip=15pt

\title[Deligne--Hitchin moduli space]{Torelli theorem for
the Deligne--Hitchin moduli space}

\author[I. Biswas]{Indranil Biswas}

\address{School of Mathematics, Tata Institute of Fundamental
Research, Homi Bhabha Road, Bombay 400005, India}

\email{indranil@math.tifr.res.in}

\author[T. L. G\'omez]{Tom\'as L. G\'omez}

\address{Instituto de Ciencias Matem\'aticas (CSIC-UAM-UC3M-UCM), 
Serrano 113bis, 28006 Madrid, Spain; and
Facultad de Ciencias Matem\'aticas, 
Universidad Complutense de Madrid, 28040 Madrid, Spain}

\email{tomas.gomez@mat.csic.es}

\author[N. Hoffmann]{Norbert Hoffmann}

\address{Freie Universit\"at Berlin, Institut f\"ur Mathematik, 
Arnimallee 3, 14195 Berlin, Germany;
Universit\"at G\"ottingen, Mathematisches Institut,
Bunsenstra{\ss}e 3-5, 37073 G\"ottingen, Germany}

\email{nhoffman@mi.fu-berlin.de; hoffmann@uni-math.gwdg.de}

\author[M. Logares]{Marina Logares}

\address{Departamento de Matematica Pura,
Facultade de Ciencias, Rua do Campo Alegre 687,
4169-007 Porto Portugal}

\email{mlogares@fc.up.pt}



\date{}

\begin{abstract}
Fix integers $g\, \geq\, 3$ and $r\, \geq\,2$, with
$r\, \geq\,3$ if $g\, =\,3$. Given a compact
connected Riemann surface $X$ of genus $g$, let
$\MDH(X)$ denote the corresponding
$\text{SL}(r, {\mathbb C})$ Deligne--Hitchin
moduli space. We prove that
the complex analytic space $\MDH(X)$
determines (up to an isomorphism) the unordered pair
$\{X, \overline{X}\}$, where $\overline{X}$ is the
Riemann surface defined by the opposite almost
complex structure on $X$.
\end{abstract}

\maketitle

\section{Introduction}

Let $X$ be a compact connected Riemann surface of genus
$g$, with $g\,\geq\, 2$. We denote by $X_{\mathbb{R}}$
the $C^\infty$ real manifold of dimension two underlying $X$.
Let $\overline{X}$ be the Riemann surface defined by the
almost complex structure $-J_X$ on $X_{\mathbb{R}}$;
here $J_X$ is the almost complex structure of $X$.

Fix an integer $r\,\geq\, 2$.
The main object of this paper is the
$\text{SL}(r, {\mathbb C})$ Deligne--Hitchin moduli space
\begin{equation*}
\MDH(X)\, =\, \MDH(X\, , \text{SL}(r, {\mathbb C}))
\end{equation*}
associated to $X$. This moduli space $\MDH(X)$
is a complex analytic space of complex dimension
$1+2(r^2-1)(g-1)$, which comes with a natural
surjective holomorphic map
\begin{equation*}
\MDH(X) \longrightarrow \CP^1 \,
=\, {\mathbb C}\cup\{\infty\}.
\end{equation*}
We briefly recall from \cite[page 7]{Si1} the
description of $\MDH(X)$ (in \cite{Si1}, the group
$\text{GL}(r, {\mathbb C})$ is considered instead of
$\text{SL}(r, {\mathbb C})$).
\begin{itemize}
 \item The fiber of $\MDH(X)$ over
$\lambda \,=\, 0 \,\in \,\mathbb{C}\,\subset\,\CP^1$
is the moduli space $\MHiggs(X)$ of semistable
$\text{SL}(r,{\mathbb C})$ Higgs bundles
$(E, \theta)$ over $X$ (see section
\ref{sec:Higgs} for details).
 \item The fiber of $\MDH(X)$ over any
 $\lambda \,\in\, \mathbb{C}^* \subset \CP^1$ is
 canonically biholomorphic to the moduli space
 $\Mconn(X)$ of holomorphic $\text{SL}(r,
 {\mathbb C})$ connections $(E, \nabla)$
 over $X$. In fact the restriction of
 $\MDH(X)$ to $\mathbb{C} \subset \CP^1$
 is the moduli space
 \begin{equation*}
 \Mlambda( X) \longrightarrow \mathbb{C}
 \end{equation*}
 of $\lambda$--connections over $X$ for the
 group $\text{SL}(r, {\mathbb C})$ (see
 section \ref{sec:lambda} for details).
 \item The fiber of $\MDH(X)$ over $\lambda \,=\,
 \infty \,\in\, \CP^1$ is the moduli space
 $\MHiggs(\overline{X})$ of semistable
 $\text{SL}(r,{\mathbb C})$ Higgs bundles
 over $\overline{X}$. Indeed, the complex analytic
 space $\MDH(X)$ is constructed by glueing
 $\Mlambda(X)$ to the analogous moduli space
 \begin{equation*}
 \Mlambda( \overline{X}) \longrightarrow \mathbb{C}
 \end{equation*}
 of $\lambda$--connections over $\overline{X}$. 
 One identifies the fiber of $\Mlambda(X)$
 over $\lambda \in \mathbb{C}^*$ with the fiber
 of $\Mlambda( \overline{X})$ over $1/\lambda \in
 \mathbb{C}^*$; the identification is done using the
fact that the holomorphic connections
 over both $X$ and $\overline{X}$ correspond
 to representations of $\pi_1(X_{\mathbb{R}})$ in
$\text{SL}(r, {\mathbb C})$
 (see section \ref{sec:MDH} for details).
\end{itemize}

This construction of $\MDH(X)$ is due to Deligne \cite{De}.
In \cite{Hi2}, Hitchin constructed the twistor space
for the hyper--K\"ahler structure of the moduli space
$\MHiggs(X)$; the complex analytic space $\MDH(X)$ is
identified with this twistor space (see \cite[page 8]{Si1}).

We note that while both $\Mlambda(X)$ and
$\Mlambda(\overline{X})$ are complex algebraic
varieties, the moduli space $\MDH(X)$
does not have any natural algebraic structure.

If we replace $X$ by $\overline{X}$, then the isomorphism
class of the Deligne--Hitchin moduli space clearly remains
unchanged. In fact, there is a canonical holomorphic isomorphism
of $\MDH(X)$ with $\MDH(\overline{X})$
over the automorphism of ${\mathbb C}{\mathbb P}^1$ defined
by $\lambda\, \longmapsto\, 1/\lambda$.

We prove the following theorem (see Theorem \ref{thm1}):
\begin{theorem}\label{thm0}
Assume that $g\,\geq\,3$, and if $g\,=\,3$, then assume that
$r\, \geq\, 3$.
The isomorphism class of the complex analytic space
$\MDH(X)$ determines uniquely the isomorphism
class of the unordered pair of Riemann
surfaces $\{X\, ,\overline{X}\}$.
\end{theorem}
In other words, if $\MDH(X)$ is
biholomorphic to the Deligne--Hitchin moduli
space $\MDH(Y)$ for another compact
connected Riemann surface $Y$, then either
$Y\, \cong\, X$ or $Y\, \cong \, \overline{X}$.

This paper is organized as follows.
Higgs bundles are dealt with in
Section \ref{sec:Higgs};
we also obtain a Torelli theorem for them
(see Corollary \ref{torelli:Higgs}).
The $\lambda$--connections are considered
in Section \ref{sec:lambda}, which also
contains a Torelli theorem for their moduli
space (see Corollary \ref{torelli:lambda}).
Finally, Section \ref{sec:MDH} deals with
the Deligne--Hitchin moduli space; here we
prove our main result.

\section{Higgs bundles}\label{sec:Higgs}

Let $X$ be a compact connected Riemann surface of genus
$g$, with $g\,\geq\, 3$. Fix an integer $r\,\geq\, 2$.
If $g\, =\,3$, then we assume that $r\,\geq\, 3$. Let
\begin{equation}\label{Mbun-def}
\Mbun
\end{equation}
be the moduli space
of semistable $\text{SL}(r,{\mathbb C})$--bundles on $X$.
So $\Mbun$ parameterizes all $S$--equivalence
classes of semistable vector bundles $E$ over $X$ of rank $r$
together with an isomorphism
$\bigwedge^r E\, \cong \,{\mathcal O}_X$. The moduli space
$\Mbun$ is known to be an irreducible normal
complex projective variety of dimension $(r^2-1)(g-1)$.
Let
\begin{equation}\label{Mbun-in}
\Mbun^{\stab}\, \subset\, \Mbun
\end{equation}
be the open subvariety parameterizing stable
$\text{SL}(r,{\mathbb C})$ bundles on $X$. This open subvariety
coincides with the smooth locus of $\Mbun$ according to
\cite[page 20, Theorem 1]{NR1}.

\begin{lemma} \label{no_oneforms}
 The holomorphic cotangent bundle
 \begin{equation*}
 T^* \Mbun^{\stab} \longrightarrow \Mbun^{\stab}
 \end{equation*}
 does not admit any nonzero holomorphic section.
\end{lemma}
\begin{proof}
Fix a point $x_0\, \in\, X$, and consider the Hecke correspondence
\begin{equation*}
 \Mbun^{\stab} \stackrel{q}{\longleftarrow} {\mathcal P}
 \stackrel{p}{\longrightarrow} {\mathcal U}\, \subseteq\, \Mbunhat
\end{equation*}
defined as follows:
\begin{itemize}
 \item $\Mbunhat$ denotes the moduli space of
 stable vector bundles $F$ over $X$ of rank $r$
 together with an isomorphism
 $\bigwedge^r F\, \cong \, {\mathcal O}_X(x_0)$.
 \item ${\mathcal U}\, \subseteq\, \Mbunhat$ denotes
 the locus of all $F$ for which every subbundle
  $F' \subset F$ with
  $0\, <\, \text{rank}(F') \, <\, r$ has
  negative degree; such vector bundles $F$ are called
  $(0\, ,1)$--\textit{stable} (see \cite[page 306,
  Definition 5.1]{NR2}, \cite[page 563]{BBGN}).
 \item $p: {\mathcal P} \longrightarrow {\mathcal U}$
  is the $\mathbb{P}^{r-1}$--bundle whose fiber over
  any vector bundle $F \,\in\, {\mathcal U}$ parameterizes
  all hyperplanes $H$ in the fiber $F_{x_0}$.
 \item $q: {\mathcal P} \, \longrightarrow\, \Mbun^{\stab}$
  sends any vector bundle $F \in {\mathcal U}$
  and hyperplane $H \subseteq F_{x_0}$ to the
  vector bundle $E$ given by the short exact sequence
  $$
  0\,\longrightarrow\, E\,\longrightarrow\, F
  \,\longrightarrow\, F_{x_0}/H \,\longrightarrow\, 0\,
  $$
  of coherent sheaves on $X$; here the quotient
  sheaf $F_{x_0}/H$ is supported at $x_0$.
\end{itemize}

As $\Mbunhat$ is a smooth unirational projective
variety (see \cite[page 53]{Se}),
it does not
admit any nonzero holomorphic $1$--form.
The subset ${\mathcal U}\, \subseteq\, \Mbunhat$
is open due to \cite[page 563, Lemma 2]{BBGN},
and the conditions on $r$ and $g$ ensure that
the codimension of the complement
$\Mbunhat \setminus {\mathcal U}$ is at least
two. Hence also
\begin{equation*}
  H^0( \mathcal{U}, T^* \mathcal{U}) \,=\, 0
\end{equation*}
due to Hartog's theorem.
Since $H^0( \mathbb{P}^{r-1},\, T^* \mathbb{P}^{r-1})
\,=\, 0$,
any relative
holomorphic $1$--form on the $\mathbb{P}^{r-1}$--bundle
$p: {\mathcal P}\, \longrightarrow\, \mathcal U$
vanishes identically. Thus we conclude that
\begin{equation*}
  H^0( \mathcal{P}, T^* \mathcal{P}) \,=\, 0\, .
\end{equation*}
The same follows for the variety $\Mbun^{\stab}$,
because the algebraic map $q: {\mathcal P}
\, \longrightarrow\, \Mbun^{\stab}$
is dominant.
\end{proof}

We denote by $K_X$ the canonical line bundle on $X$. Let
\begin{equation*}
\MHiggs(X)\, =\, \MHiggs(X\, , \text{SL}(r,{\mathbb C}))
\end{equation*}
denote the moduli space of semistable $\text{SL}(r,{\mathbb C})$
Higgs bundles over $X$. So $\MHiggs(X)$ parameterizes
all $S$--equivalence classes of semistable pairs
$(E\, ,\theta)$ consisting of a vector bundle $E$ over
$X$ of rank $r$ together with an isomorphism
$\bigwedge^r E\, \cong \,{\mathcal O}_X$, and a Higgs
field $\theta: E \longrightarrow E \otimes K_X$ with
$\text{trace}(\theta)\,=\, 0$. The moduli space $\MHiggs(X)$ is an
irreducible normal complex algebraic variety of dimension
$2(r^2-1)(g-1)$ according to \cite[page 70, Theorem 11.1]{Si3}.

There is a natural embedding
\begin{equation}\label{e3}
\iota\,:\, \Mbun\,\hookrightarrow\,\MHiggs(X)
\end{equation}
defined by $E\,\longmapsto\, (E\, ,0)$.
Let
\begin{equation*}
  \MHiggs^{\stab}(X) \,\subset\, \MHiggs(X)
\end{equation*}
be the Zariski
open locus of Higgs bundles $(E, \theta)$
whose underlying vector bundle $E$ is stable
(openness of $\MHiggs^{\stab}(X)$ follows from
\cite[page 635, Theorem 2.8(B)]{Ma}). Let
\begin{equation} \label{pr_Higgs}
  \pr_E\,:\, \MHiggs^{\stab}(X)
  \,\longrightarrow\, \Mbun^{\stab}
\end{equation}
be the forgetful map defined by
$(E, \theta) \longmapsto E$, where $\Mbun^{\stab}$
is defined in \eqref{Mbun-in}.
One has a canonical isomorphism
\begin{equation} \label{higgsiso}
  \MHiggs^{\stab}(X) \,\stackrel{\sim}{\longrightarrow}\,
  T^* \Mbun^{\stab}
\end{equation}
of varieties over $\Mbun^{\stab}$,
because holomorphic cotangent vectors to a point
$E \in \Mbun^{\stab}$ correspond, via
deformation theory and Serre duality,
to Higgs fields
$\theta: E \longrightarrow E \otimes K_X$ with
$\text{trace}(\theta)\,=\, 0$.
In particular, $\MHiggs^{\stab}(X)$ is
contained in the smooth locus
\begin{equation*}
  \MHiggs(X)^{\smooth} \,\subset\, \MHiggs(X)\, .
\end{equation*}

We recall that the \textit{Hitchin map}
\begin{equation}\label{e5}
H\, :\, \MHiggs(X)\,\longrightarrow\,
\bigoplus_{i=2}^r H^0(X,\, K^{\otimes i}_X)
\end{equation}
is defined by sending each Higgs bundle
$(E, \theta)$ to the characteristic polynomial
of $\theta$ \cite{Hi1}, \cite{Hi2}.

The multiplicative group ${\mathbb C}^*$ acts on
the moduli space $\MHiggs(X)$ as follows:
\begin{equation} \label{action}
t \cdot (E\, ,\theta)\, =\, (E\, ,t \theta)\, .
\end{equation}
On the other hand, ${\mathbb C}^*$ acts on
the Hitchin space $\bigoplus_{i=2}^r
H^0(X,\, K^{\otimes i}_X)$ as
\begin{equation}\label{ext}
t \cdot (v_2\, ,\cdots\, , v_i\, ,\cdots\, , v_r)\,=\, 
(t^2 v_2\, ,\cdots\, , t^i v_i\, ,\cdots\, , t^r v_r)\, ,
\end{equation}
where $v_i\, \in\, H^0(X,\, K^{\otimes i}_X)$ and
$i\, \in\, \{2, \ldots, r\}$. The Hitchin map $H$
in \eqref{e5} intertwines these two actions
of ${\mathbb C}^*$.
Note that there is no nonzero holomorphic function
on the Hitchin space which is homogeneous
of degree $1$ for this action (a function $f$ is
homogeneous of degree $d$ if $f(t \cdot (v_2,
\cdots, v_r))\,=\, t^{d} f((v_2,
\cdots, v_r))$), because all the exponents of $t$
in \eqref{ext} are at least two.

\begin{lemma} \label{no_vectorfields}
  The holomorphic tangent bundle
  \begin{equation*}
    T \Mbun^{\stab} \,\longrightarrow \,\Mbun^{\stab}
  \end{equation*}
  does not admit any nonzero holomorphic section.
\end{lemma}

\begin{proof}
  The proof of \cite[page 110, Theorem 6.2]{Hi1}
  carries over to this situation as follows.
  A holomorphic section $s$ of
  $T \Mbun^{\stab}$ provides (by
  contraction) a holomorphic function
  \begin{equation}\label{f}
  f\, :\, T^* \Mbun^{\stab}\,\longrightarrow\,\mathbb C
  \end{equation}
  on the total space of the cotangent bundle
  $T^*\Mbun^{\stab}$, which is linear on the fibers.
  Under the isomorphism in \eqref{higgsiso},
  it corresponds to a function on $\MHiggs^{\stab}( X)$.
The conditions on $g$ and $r$ imply that the complement of
  $\MHiggs^{\stab}( X)$ has codimension at least two
  in $\MHiggs(X)$. Since the latter is normal, the function
  $f$ in \eqref{f} extends to a holomorphic
  function
  \begin{equation*}
    \widetilde{f}\,:\, \MHiggs(X) \,\longrightarrow \,
\mathbb C\, ,
  \end{equation*}
  for example by \cite[page 90, Korollar 2]{Sc}.
  Since $f$ is linear on the fibers, we know that
$\widetilde{f}$ is homogeneous of degree $1$ for the action
\eqref{action} of $\mathbb{C}^*$.

  On the moduli space $\MHiggs(X)$, the Hitchin map
  \eqref{e5} is proper \cite[Theorem 6.1]{Ni},
  and also its fibers are connected. Therefore,
  the function $\widetilde{f}$ is constant
  on the fibers of the Hitchin map. Hence
  $\widetilde{f}$ comes from a holomorphic
  function on the Hitchin space, which is
  still homogeneous of degree $1$. We noted earlier
that there are no nonzero holomorphic
  functions on the Hitchin space which are
homogeneous of degree $1$. Therefore, $\widetilde{f} = 0$, and
consequently we have $f \,=\, 0$ and $s \,=\, 0$.
\end{proof}

\begin{corollary} \label{cor:Higgs}
  The restriction of the holomorphic tangent bundle
  \begin{equation*}
    T \MHiggs(X)^{\smooth}
    \,\longrightarrow\, \MHiggs(X)^{\smooth}
  \end{equation*}
  to $\iota( \Mbun^{\stab})\,\subset\,\MHiggs(X)^{\smooth}$
  does not admit any nonzero holomorphic section.
\end{corollary}

\begin{proof}
  Using Lemma \ref{no_vectorfields}, it
  suffices to show that the normal bundle
  of the embedding
  \begin{equation*}
    \iota\,:\, \Mbun^{\stab} \,\hookrightarrow\,
    \MHiggs(X)^{\smooth} 
  \end{equation*}
  has no nonzero holomorphic sections. The isomorphism in
  \eqref{higgsiso} allows us to identify this normal
  bundle with $T^* \Mbun^{\stab}$. Now the assertion
  follows from Lemma \ref{no_oneforms}.
\end{proof}

The next step is to show that the above property
uniquely characterizes the subvariety
$\iota( \Mbun) \subset \MHiggs( X)$.
This will follow from the following proposition.

\begin{proposition} \label{fix_higgs}
  Let $Z$ be an irreducible component
  of the fixed point locus
  \begin{equation} \label{fixedlocus}
    \MHiggs( X)^{\mathbb{C}^*}\, \subseteq\, \MHiggs( X)\, .
  \end{equation}
  Then $\dim(Z) \,\leq\, (r^2-1)(g-1)$, with equality
  only for $Z \,=\, \iota( \Mbun)$.
\end{proposition}

\begin{proof}
  The ${\mathbb{C}^*}$--equivariance
  of the Hitchin map $H$ in \eqref{e5} implies
  \begin{equation*}
    \MHiggs( X)^{\mathbb{C}^*} \,\subseteq\, H^{-1}( 0),
  \end{equation*}
  because $0$ is the only fixed point in the Hitchin
  space. We recall that $H^{-1}(0)$ is called the
  \textit{nilpotent cone}. The irreducible components
  of $H^{-1}(0)$ are parameterized by the conjugacy
  classes of the nilpotent elements in the Lie algebra
  $\text{sl}(r,{\mathbb C})$, and each irreducible
  component of $H^{-1}(0)$ is of dimension
  $(r^2-1)(g-1)$ \cite{La}.

  Thus $\dim(Z) \,\leq\, (r^2-1)(g-1)$, and if equality
  holds, then $Z$ is an irreducible component of
  the nilpotent cone $H^{-1}(0)$. A result due to
  Simpson, \cite[page 76, Lemma 11.9]{Si3}, implies that the
  only irreducible component of $H^{-1}( 0)$
  contained in the fixed point locus $\MHiggs( X)^{\mathbb{C}^*}$
  defined in \eqref{fixedlocus} is the image
  $\iota( \Mbun)$ of the embedding in \eqref{e3}.
\end{proof}

\begin{corollary} \label{torelli:Higgs}
  The isomorphism class of the complex analytic space
  $\MHiggs(X)$ determines uniquely the isomorphism
  class of the Riemann surface $X$, meaning if
$\MHiggs(X)$ is biholomorphic to
$\MHiggs(Y)$ for another compact connected Riemann
surface $Y$ of the same genus $g$, then $Y \,\cong\, X$.
\end{corollary}

\begin{proof}
Let $Z \,\subset\, \MHiggs(X)$ be a closed analytic
subset with the following three properties:
\begin{itemize}
 \item $Z$ is irreducible and has
  complex dimension $(r^2-1)(g-1)$.
 \item The smooth locus $Z^{\smooth} \,\subseteq\, Z$
  lies in the smooth locus
  $\MHiggs( X)^{\smooth} \subset \MHiggs( X)$.
 \item The restriction of the holomorphic tangent
  bundle $T \MHiggs( X)^{\smooth}$ to the subspace
  $Z^{\smooth} \,\subset\, \MHiggs( X)^{\smooth}$
  has no nonzero holomorphic sections.
\end{itemize}
By Corollary \ref{cor:Higgs}, the image
$\iota( \Mbun)$ of the embedding $\iota$
in \eqref{e3} has these properties.

The action \eqref{action} of $\mathbb{C}^*$
on $\MHiggs( X)$ defines a holomorphic vector field
\begin{equation*}
  \MHiggs(X)^{\smooth} \longrightarrow
  T \MHiggs(X)^{\smooth}.
\end{equation*}
The third assumption on $Z$ says that any holomorphic vector
field on $\MHiggs(X)^{\smooth}$ vanishes on $Z^{\smooth}$.
Therefore, it follows that the stabilizer of each point
in $Z^{\smooth} \,\subset\, \MHiggs( X)$ has
nontrivial tangent space at $1 \in \mathbb{C}^*$, and hence
the stabilizer must be the full group $\mathbb{C}^*$.

This shows that the fixed point locus
$\MHiggs( X)^{\mathbb{C}^*} \,\subseteq\, \MHiggs( X)$
contains $Z^{\smooth}$, and hence also
contains its closure $Z$ in $\MHiggs( X)$.
Due to Proposition \ref{fix_higgs}, this can
only happen for $Z \,=\, \iota( \Mbun)$.
In particular, we have $Z \,\cong\, \Mbun$.

We have just shown that the isomorphism class of
$\MHiggs(X)$ determines the isomorphism
class of $\Mbun$. The latter determines the
isomorphism class of $X$ due to a theorem of
Kouvidakis and Pantev \cite[page 229, Theorem E]{KP}.
\end{proof}

\begin{remark} \upshape
In \cite{BG}, an analogous Torelli theorem is
proved for Higgs bundles $(E\, , \theta)$ such that
the rank and the degree of the underlying
vector bundle $E$ are coprime.
\end{remark}

\section{The $\lambda$--connections} \label{sec:lambda}

In this section, we consider vector bundles
with connections, and more generally with
$\lambda$--connections in the sense of
\cite[page 87]{Si2} and \cite[page 4]{Si1}.
We denote by
\begin{equation*}
  \Mlambda(X)\, =\, \Mlambda(X\, , \text{SL}(r, {\mathbb C}))
\end{equation*}
the moduli space of triples of the form
$(\lambda\, , E\, ,\nabla)$, where
$\lambda$ is a complex number, and $(E\, ,\nabla)$
is a $\lambda$--connection on $X$ for the group
$\text{SL}(r, {\mathbb C})$. We recall that given
any $\lambda\, \in\, \mathbb C$, a $\lambda$--\textit{connection}
on $X$ for the group
$\text{SL}(r, {\mathbb C})$ is a pair $(E\, ,\nabla)$, where
\begin{itemize}
\item $E\, \longrightarrow\, X$ is a holomorphic vector bundle
of rank $r$ together with an isomorphism
$\bigwedge^r E\, \cong \, {\mathcal O}_X$.

\item $\nabla\, :\, E\, \longrightarrow\, E\otimes K_X$
is a $\mathbb C$--linear homomorphism of sheaves
satisfying the following two conditions:
\begin{enumerate}
\item If $f$ is a locally defined holomorphic function
on $\mathcal{O}_X$ and $s$ is a locally defined
holomorphic section of $E$, then
$$
\nabla(fs) \, =\, f\cdot \nabla(s) + \lambda \cdot s\otimes df\, .
$$
\item The operator $\bigwedge^rE\, \longrightarrow\, 
(\bigwedge^rE)\otimes K_X$ induced by $\nabla$
coincides with $\lambda \cdot d$.
\end{enumerate}
\end{itemize}

The moduli space $\Mlambda(X)$ is a complex algebraic
variety of dimension $1+2(r^2-1)(g-1)$. It is equipped with
a surjective algebraic morphism
\begin{equation} \label{pr_lambda}
\pr_{\lambda}\,:\, \Mlambda(X)\, \longrightarrow\, \mathbb C
\end{equation}
defined by $(\lambda, E, \nabla) \,\longmapsto\, \lambda$.

A $0$--connection is a Higgs bundle, so
\begin{equation*}
  \MHiggs( X) \,=\, \pr_{\lambda}^{-1}( 0) \subset \Mlambda(X)
\end{equation*}
is the moduli space of Higgs bundles considered
in the previous section. In particular, the
embedding \eqref{e3} of $\Mbun$ into $\MHiggs( X)$
also gives an embedding of $\Mbun$ into $\Mlambda(X)$.
Slightly abusing notation, we denote
this embedding again by
\begin{equation} \label{iota}
\iota\,:\, \Mbun\,\hookrightarrow\,\Mlambda(X)\, .
\end{equation}
It maps the stable locus $$\Mbun^{\stab} \,\subset\,
\Mbun$$ into the smooth locus
\begin{equation}\label{hsl}
\Mlambda(X)^{\smooth}\,\subset\, \Mlambda(X)\, .
\end{equation}

We let $\mathbb{C}^*$ act on $\Mlambda( X)$ as
\begin{equation} \label{extended_action}
t \cdot (\lambda, E, \nabla) \,=\,
(t \cdot \lambda, E, t \cdot \nabla)\, .
\end{equation}
This extends the $\mathbb{C}^*$ action on $\MHiggs( X)$
introduced above in formula \eqref{action}.

\begin{proposition} \label{fix_lambda}
  Let $Z$ be an irreducible component
  of the fixed point locus
  \begin{equation*}
    \Mlambda( X)^{\mathbb{C}^*} \,\subseteq\, \Mlambda( X).
  \end{equation*}
  Then $\dim(Z) \,\leq\, (r^2-1)(g-1)$, with equality
  only for $Z \,=\, \iota( \Mbun)$.
\end{proposition}
\begin{proof}
  A point $(\lambda, E, \nabla) \,\in\, \Mlambda( X)$
  can only be fixed by $\mathbb{C}^*$ if $\lambda \,=\, 0$.
  Hence $Z$ is automatically contained in $\MHiggs(X)$.
  Now the claim follows from Proposition \ref{fix_higgs}.
\end{proof}

A $1$--connection is a
holomorphic connection in the usual sense, so
\begin{equation} \label{Mconn}
  \Mconn(X)\,:=\, \pr_{\lambda}^{-1}( 1)\,\subset\, \Mlambda( X)
\end{equation}
is the moduli space of $\text{SL}(r, {\mathbb C})$
holomorphic connections $(E, \nabla)$ over $X$. We denote by
\begin{equation*}
  \Mconn^{\stab}( X) \,\subset\, \Mconn(X)
  \qquad\text{and}\qquad
  \Mlambda^{\stab}( X) \,\subset\, \Mlambda(X)
\end{equation*}
the Zariski open subvarieties where the underlying
vector bundle $E$ is stable (openness follows from
\cite[page 635, Theorem 2.8(B)]{Ma}).

\begin{proposition} \label{no_connection}
The forgetful map
\begin{equation}\label{psi}
  \pr_E\,:\, \Mconn^{\stab}(X)\, \longrightarrow\,
  \Mbun^{\stab}
\end{equation}
defined by $(E\, , \nabla) \longmapsto E$
admits no holomorphic section.
\end{proposition}
\begin{proof}
This map $\pr_E$ is surjective,
because  a criterion due to Atiyah and Weil
implies that every stable vector bundle $E$
on $X$ of degree zero admits a holomorphic connection.
In fact, $E$ admits a unique unitary
holomorphic connection according to a theorem of
Narasimhan and Seshadri \cite{NS};
this defines a canonical $C^{\infty}$ section
\begin{equation} \label{section}
  \Mbun^{\stab} \,\longrightarrow \,\Mconn^{\stab}(X)
\end{equation}
of the map $\pr_E$. Since any two
holomorphic $\text{SL}(r, {\mathbb C})$--connections
on $E$ differ by a Higgs field
$\theta\,:\, E \,\longrightarrow E\, \otimes K_X$
with $\mathrm{trace}( \theta) \,=\, 0$,
the map $\pr_E$ in \eqref{psi} is a
holomorphic torsor under the holomorphic
cotangent bundle
$T^* \Mbun^{\stab}\,\longrightarrow\, \Mbun^{\stab}$.

  Given a complex manifold $\calM$, we denote by
  $\Treal \calM$ the tangent bundle of the
  underlying real manifold $\calM_{\mathbb{R}}$,
  and by
  \begin{equation*}
    J_{\calM}\,:\, \Treal \calM
    \,\longrightarrow\, \Treal \calM
  \end{equation*}
  the almost complex structure of $\calM$. Let
  \begin{equation} \label{torsor}
    \varpi\,:\, \calX\, \longrightarrow \,\calM
  \end{equation}
  be a holomorphic torsor under a holomorphic
  vector bundle $\calV \longrightarrow \calM$.
  To each $C^\infty$ section
  $s\,:\, \calM \,\longrightarrow\, \calX$ of $\varpi$,
  we can associate a $(0, 1)$--form
  \begin{equation*}
    \overline{\partial} s \,\in\, C^\infty( \calM,\,
    \Omega^{0, 1} \calM \otimes \calV)
  \end{equation*}
 in the following way. The vector bundle homomorphism
  \begin{equation*}
    \widetilde{ds} \,:=\, ds + J_{\calX} \circ ds
    \circ J_{\calM}\,:\, \Treal \calM \,\longrightarrow\,
    s^* \Treal \calX
  \end{equation*}
  satisfies the identity
  \begin{equation} \label{type01}
    J_{\calX} \circ \widetilde{ds} +
    \widetilde{ds} \circ J_{\calM} \,=\,
    J_{\calX} \circ ds - ds \circ J_{\calM}
    - J_{\calX} \circ ds + ds \circ J_{\calM}
    \,=\, 0\, ,
  \end{equation}
  and, since $\varpi$ is holomorphic, we also have
  \begin{equation} \label{vertical}
    d \varpi \circ \widetilde{ds} \,=\, d \varpi \circ ds
    + J_{\calM} \circ d \varpi \circ ds \circ J_{\calM}
    \,=\, \mathrm{id} - \mathrm{id} \,=\, 0\, .
  \end{equation}

  The equation in \eqref{vertical} means that
  $\widetilde{ds}$ maps
  into the subbundle of vertical tangent
  vectors in $s^* \Treal \calX$, which
  is canonically isomorphic to $\calV_{\mathbb{R}}$ (the
real vector bundle underlying the complex vector bundle
$\calV$).
  Thus we can consider $\widetilde{ds}$ as a real
  $1$--form
  \begin{equation*}
    \widetilde{ds} \,\in\, C^\infty(
    \calM, \, \Treal^* \calM \otimes \calV_{\mathbb{R}})\, .
  \end{equation*}

Identify $\Treal \calM$ with $T^{0,1}{\calM}$ using the
$\mathbb R$--linear isomorphism defined by
$$
v\, \longmapsto\, v- \sqrt{-1}{\cdot}J_{\calM}(v)\, ,
$$
and also identify $\calV_{\mathbb{R}}$ with $\calV$
using the identity map. From \eqref{type01} it
follows that $\widetilde{ds}$ is actually a
$\mathbb C$--linear
homomorphism from $T^{0,1}{\calM}$ to $\calV$ in terms
of these identifications. Let
\begin{equation*}
    \overline{\partial} s \,\in\, C^\infty(
    \calM,\, \Omega^{0, 1}_{\calM} \otimes \calV)\, .
  \end{equation*}
be the $(0\, ,1)$--form with values in $\calV$ defined
by $\widetilde{ds}$. From the construction of
$\overline{\partial} s$ it is clear that
\begin{itemize}
\item $\overline{\partial} s$ vanishes if and
only if $s$ is holomorphic, and

\item $\overline{\partial} s$ is $\overline{\partial}$--closed.
\end{itemize}
Therefore, $\overline{\partial} s$ defines a
Dolbeault cohomology class
\begin{equation}\label{dcc}
    [\varpi] \,:=\,
    [\overline{\partial} s] \,\in\,
    H_{\overline{\partial}}^{0, 1}( \calM,\, \calV) 
    \,\cong\, H^1( \calM,\, \calV)\, .
\end{equation}

Since $\calV$ acts on $\varpi\,:\, \calX\,\longrightarrow\,
\calM$, each section $v \,\in\, C^{\infty}( \calM, \,\calV)$
  acts on the sections of $\varpi$; we denote this
  action by $s \,\longmapsto\, v + s$.
  The above construction implies that
  \begin{equation}\label{if}
    \overline{\partial} (v + s)
    \,=\, \overline{\partial} v
    + \overline{\partial} s\, .
  \end{equation}
Consequently, the Dolbeault cohomology class $[\varpi]$
in \eqref{dcc} does not depend on the choice of the
$C^{\infty}$ section $s$. From \eqref{if} it also follows that 
$[\varpi]$ vanishes if and only if the torsor $\varpi$
in \eqref{torsor} admits a holomorphic section.

  We now take $\varpi$ to be the torsor $\pr_E$
in \eqref{psi} under the cotangent bundle $T^* \Mbun^{\stab}$,
  and we take $s$ to be the $C^\infty$ section in
  \eqref{section}. For this case, the class
  \begin{equation} \label{torsorclass}
    [\overline{\partial} s] \,\in\,
    H^1( \Mbun^{\stab}, \,T^* \Mbun^{\stab})
  \end{equation}
  has been computed in
  \cite[page 308, Theorem 2.11]{BR}; the result is
  that it is a nonzero multiple of $c_1(\Theta)$,
  where $\Theta$ is the ample generator of
  $\text{Pic}(\Mbun^{\stab})$. In particular,
  the cohomology class \eqref{torsorclass}
  of the torsor $\pr_E$ in question is nonzero.
Therefore, $\pr_E$ does not admit any holomorphic section.
\end{proof}

We note that the forgetful map $\pr_E$ defined in
Proposition \ref{no_connection} extends
canonically from $\Mconn^{\stab}(X)$ to
$\Mlambda^{\stab}(X)$. Slightly abusing
notation, we denote this extended map
again by
\begin{equation*}
  \pr_E\, :\, \Mlambda^{\stab}(X)\, \longrightarrow\,
  \Mbun^{\stab}\, .
\end{equation*}
This map
is defined by $(\lambda, E, \nabla) \,\longmapsto\, E$,
and it also extends the map $\pr_E$ in \eqref{pr_Higgs}.

\begin{corollary} \label{no_lambdaconnection}
  The only holomorphic map
  \begin{equation*}
    s\,:\, \Mbun^{\stab} \,\longrightarrow \,\Mlambda^{\stab}(X)
  \end{equation*}
  with $\pr_E \circ s \,=\, \mathrm{id}$ is the restriction
  \begin{equation*}
  \iota\,:\, \Mbun^{\stab}\,\hookrightarrow\,\Mlambda^{\stab}(X)
  \end{equation*}
  of the embedding $\iota$ defined in \eqref{iota}.
\end{corollary}

\begin{proof}
  The composition
  \begin{equation*}
    \Mbun^{\stab} \,\stackrel{s}{\longrightarrow}\,
    \Mlambda^{\stab}(X) 
    \,\stackrel{\pr_{\lambda}}{\longrightarrow}\, 
    \mathbb{C}\, ,
  \end{equation*}
where $\pr_{\lambda}$ is the projection in \eqref{pr_lambda},
  is a holomorphic function on $\Mbun^{\stab}$,
  and hence it is a constant function.
Up to the $\mathbb{C}^*$ action
  in \eqref{extended_action}, we may assume that
  this constant is either $0$ or it is $1$.

  If this constant were $1$, then $s$ would
  factor through $\pr_{\lambda}^{-1}( 1) \,=\,
  \Mconn^{\stab}( X)$, which would contradict
  Proposition \ref{no_connection}.

  Hence this constant is $0$, and $s$ factors
  through $\pr_{\lambda}^{-1}(0) \,=\, \MHiggs^{\stab}(X)$.
  Thus $s$ corresponds, under the isomorphism
  \eqref{higgsiso}, to a holomorphic global
  section of the vector bundle
  $T^* \Mbun^{\stab}$. But any such section
  vanishes due to Lemma \ref{no_oneforms};
  this means that $s$ is indeed the restriction
  of the canonical embedding $\iota$ in \eqref{iota}.
\end{proof}

\begin{corollary} \label{cor:lambda}
As in \eqref{hsl}, let $\Mlambda(X)^{\smooth}$ be the
smooth locus of $\Mlambda(X)$. 
The restriction of the holomorphic tangent bundle
  \begin{equation*}
    T \Mlambda(X)^{\smooth} \,\longrightarrow\,
    \Mlambda(X)^{\smooth}
  \end{equation*}
  to $\iota( \Mbun^{\stab}) \,\subset \,\Mlambda(X)^{\smooth}$
  does not admit any nonzero holomorphic section.
\end{corollary}

\begin{proof}
  We denote the holomorphic normal
  bundle of the restricted embedding
  \begin{equation*}
    \iota\,:\, \Mbun^{\stab} \,\hookrightarrow\,
    \Mlambda(X)^{\smooth}  
  \end{equation*}
  by $\mathcal{N}$. Due to Lemma
  \ref{no_vectorfields}, it suffices
  to show that this vector bundle
  $\mathcal{N}$ over $\Mbun^{\stab}$
  has no nonzero holomorphic sections.

 One has a canonical isomorphism
  \begin{equation}\label{ci}
    \Mlambda^{\stab}(X) \,\stackrel{\sim}{\longrightarrow}\,
    \mathcal{N}
  \end{equation}
  of varieties over $\Mbun^{\stab}$, defined by sending any
  $(\lambda, E, \nabla)$ to the derivative at $t=0$
  of the map
  \begin{equation*}
  {\mathbb C}\, \longrightarrow\, \Mlambda(X)\, , \qquad
  t \longmapsto (t \cdot \lambda\, , E \, ,t \cdot \nabla)\, .
  \end{equation*}
  Using this isomorphism, from Corollary \ref{no_lambdaconnection}
we conclude that vector bundle $\mathcal{N}$ over $\Mbun^{\stab}$
 does not have any nonzero holomorphic sections. This completes
the proof.
\end{proof}

\begin{corollary} \label{torelli:lambda}
  The isomorphism class of the complex analytic space
  $\Mlambda(X)$ determines uniquely the isomorphism
  class of the Riemann surface $X$.
\end{corollary}

\begin{proof}
The proof is similar to that of Corollary
\ref{torelli:Higgs}.
Let $Z \,\subset\, \Mlambda(X)$ be a closed analytic
subset satisfying the following three conditions:
\begin{itemize}
 \item $Z$ is irreducible and has
  complex dimension $(r^2-1)(g-1)$.
 \item The smooth locus $Z^{\smooth} \,\subseteq\, Z$
  lies in the smooth locus $\Mlambda(X)^{\smooth} \subset \Mlambda(X)$.
 \item The restriction of the holomorphic tangent
  bundle $T \Mlambda(X)^{\smooth}$ to the subspace
  $Z^{\smooth}$ has no nonzero holomorphic sections.
\end{itemize}
{}From Corollary \ref{cor:lambda} we know that
$\iota( \Mbun)$ satisfies all these conditions.

Consider the vector field on $\Mlambda(X)^{\smooth}$ given by the
action of $\mathbb{C}^*$
on $\Mlambda(X)$ in \eqref{extended_action}. From the
third condition on $Z$ we know that this vector field vanishes
on $Z^{\smooth}$. This implies that the fixed point locus
$\Mlambda(X)^{\mathbb{C}^*}$ contains $Z^{\smooth}$, and hence
also contains its closure $Z$. Therefore, using Proposition
\ref{fix_lambda} it follows that $Z\, =\, \iota( \Mbun)$; in
particular, $Z$ is isomorphic to $\Mbun$. Finally
the isomorphism class of $X$ is recovered from the isomorphism
class of $\Mbun$ using \cite[page 229, Theorem E]{KP}.
\end{proof}

\section{The Deligne--Hitchin moduli space}
\label{sec:MDH}

We recall Deligne's construction \cite{De} of
the Deligne--Hitchin moduli space $\MDH(X)$, as
described in \cite[page 7]{Si1}.

Let $X_{\mathbb{R}}$ be the $C^\infty$ real
manifold of dimension two underlying $X$. Fix a point
$x_0\, \in\, X_{\mathbb{R}}$. Let
$$
\Mrep(X_{\mathbb{R}})\, :=\, \text{Hom}(\pi_1(X_{\mathbb{R}}, x_0),
\text{SL}(r,{\mathbb C}))/\!\!/ \text{SL}(r,{\mathbb C})\,
$$
denote the moduli space of representations
$\rho: \pi_1(X_{\mathbb{R}},x_0) \,\longrightarrow\,
\text{SL}(r,{\mathbb C})$; the group $\text{SL}(r,{\mathbb C})$
acts on $\text{Hom}(\pi_1(X_{\mathbb{R}}, x_0),
\text{SL}(r,{\mathbb C}))$ through the adjoint action of
$\text{SL}(r,{\mathbb C})$ on itself.
Since the fundamental groups for
different base points are identified up to an inner automorphism,
the space $\Mrep(X_{\mathbb{R}})$ is independent of the
choice of $x_0$. Hence we will omit any reference to $x_0$.

The Riemann--Hilbert correspondence
defines a biholomorphic isomorphism
\begin{equation} \label{RH}
\Mrep(X_{\mathbb{R}})
\, \stackrel{\sim}{\longrightarrow}\,
\Mconn(X)\, .
\end{equation}
It sends a representation $\rho: \pi_1(X_{\mathbb{R}})
\longrightarrow \text{SL}(r,{\mathbb C})$ to the
associated holomorphic $\text{SL}(r,{\mathbb C})$--bundle
$E_{\rho}^X$ over $X$, endowed with the induced
connection $\nabla_{\rho}^X$. The inverse of \eqref{RH}
sends a connection to its monodromy representation,
which makes sense because any holomorphic connection
on a Riemann surface is automatically flat.

Given $\lambda \in \mathbb{C}^*$, we can
similarly associate to a representation
$$\rho\,:\, \pi_1(X_{\mathbb{R}}) \,\longrightarrow\,
\text{SL}(r,{\mathbb C})$$ the $\lambda$--connection
$( E_{\rho}^X\, , \lambda \cdot \nabla_{\rho}^X)$.
This defines a holomorphic open embedding
\begin{equation}\label{BD}
  \mathbb{C}^* \times \Mrep(X_{\mathbb{R}})
 \, \longrightarrow\, \Mlambda(X)
\end{equation}
onto the open locus $\pr_{\lambda}^{-1}(
\mathbb{C}^*) \,\subset\, \Mlambda(X)$ of all
triples $(\lambda\, , E\, , \nabla)$
with $\lambda \,\neq\, 0$.

Let $J_X$ denote the almost complex structure of
the Riemann surface $X$. Then $-J_X$ is also an
almost complex structure on $X_{\mathbb{R}}$;
the Riemann surface defined by $-J_X$ will be denoted
by $\overline{X}$.

We can also consider the
moduli space $\Mlambda(\overline{X})$ of
$\lambda$--connections on $\overline{X}$, etcetera.

Now one defines the Deligne--Hitchin moduli space
\begin{equation*}
  \MDH(X) \,:=\, \Mlambda(X) \cup \Mlambda(\overline{X})
\end{equation*}
by glueing $\Mlambda(\overline{X})$
to $\Mlambda(X)$, along the image of
$\mathbb{C}^* \times \Mrep(X_{\mathbb{R}})$
for the map in \eqref{BD}.
More precisely, one identifies, for each
$\lambda \in \mathbb{C}^*$ and each
representation $\rho \in \Mrep(X_{\mathbb{R}})$,
the two points
\begin{equation*}
  (\lambda\, , E_{\rho}^X\, , \lambda \cdot \nabla_{\rho}^X)
  \,\in\, \Mlambda(X) \qquad\text{and}\qquad
  (\lambda^{-1}\, , E_{\rho}^{\overline{X}}\, , \lambda^{-1}
  \cdot \nabla_{\rho}^{\overline{X}})
  \, \in\,  \Mlambda(\overline{X})\, .
\end{equation*}
This identification yields a complex analytic space $\MDH(X)$ of
dimension $2(r^2-1)(g-1)+1$. This analytic space does not
possess a natural algebraic structure since the Riemann--Hilbert
correspondence \eqref{RH} is holomorphic and not algebraic.

The forgetful map $\pr_{\lambda}$ in \eqref{pr_lambda}
extends to a natural holomorphic morphism
\begin{equation}\label{dpr}
\pr\, :\,  \MDH(X) \,\longrightarrow\,{\mathbb C}
{\mathbb P}^1\,=\, {\mathbb C}\cup\{\infty\}
\end{equation}
whose fiber over $\lambda \in {\mathbb C}
{\mathbb P}^1$ is canonically biholomorphic to
\begin{itemize}
 \item the moduli space $\MHiggs(X)$ of
  $\text{SL}(r, {\mathbb C})$ Higgs bundles
  on $X$ if $\lambda = 0$,
 \item the moduli space $\MHiggs(\overline{X})$
  of $\text{SL}(r, {\mathbb C})$ Higgs
  bundles on $\overline{X}$ if $\lambda = \infty$, 
 \item the moduli space $\Mrep(X_{\mathbb{R}})$
  of equivalence classes of representations
$$
\text{Hom}(\pi_1(X_{\mathbb{R}}, x_0),
\text{SL}(r,{\mathbb C}))/\!\!/ \text{SL}(r,{\mathbb C})
$$
  if $\lambda \, \neq\,  0\, , \infty$.
\end{itemize}

Now we are in a position to prove the main result.

\begin{theorem}\label{thm1}
The isomorphism class of the complex analytic space
$\MDH(X)$ determines uniquely the isomorphism
class of the unordered pair of Riemann
surfaces $\{X\, ,\overline{X}\}$.
\end{theorem}

\begin{proof}
  We denote by $\MDH(X)^{\smooth} \,\subset\, \MDH(X)$
  the smooth locus, and by
  \begin{equation*}
    T \MDH(X)^{\smooth} \,\longrightarrow \,\MDH(X)^{\smooth}
  \end{equation*}
  its holomorphic tangent bundle. Since $\Mlambda(X)$
  is open in $\MDH(X)$, Corollary \ref{cor:lambda} implies
  that the restriction of $T \MDH(X)^{\smooth}$ to 
  \begin{equation} \label{subvar1}
    \iota( \Mbun^{\stab}) \,\subset\, \Mlambda(X)^{\smooth}
    \,\subset\, \MDH(X)^{\smooth}
  \end{equation}
  does not admit any nonzero holomorphic section.
  The same argument applies if we replace $X$ by
 $\overline{X}$. Since $\Mlambda(\overline{X})$
  is also open in $\MDH(X)$, the restriction of
  $T \MDH(X)^{\smooth}$ to 
  \begin{equation} \label{subvar2}
    \iota( \Mbunbar^{\stab})
    \,\subset\, \Mlambda(\overline{X})^{\smooth}
    \,\subset \,\MDH(X)^{\smooth}
  \end{equation}
  does not admit any nonzero holomorphic section
  either. Here $\Mbunbar$ is the moduli space of
  holomorphic $\text{SL}(r,{\mathbb C})$--bundles
  $E$ on $\overline{X}$, and $\iota$ denotes, as
  in \eqref{e3} and in \eqref{iota}, the
  canonical embedding of $\Mbunbar$ into $\MHiggs(\overline{X})
\,\subset\, \Mlambda(\overline{X})$ defined by
$E\, \longmapsto\, (E,0)$.

The rest of the proof is similar to that of Corollary
\ref{torelli:Higgs}.
We will extend the $\mathbb{C}^*$ action on $\Mlambda(X)$ in
 \eqref{extended_action} to $\MDH(X)$. First consider
the action of $\mathbb{C}^*$ on $\Mlambda(\overline{X})$
defined as in \eqref{extended_action} by substituting
$\overline{X}$ in place of $X$. Note that the action of
any $t\, \in\, \mathbb{C}^*$ on the open subset
$\mathbb{C}^* \times \Mrep(X_{\mathbb{R}})
 \, \longrightarrow\, \Mlambda(X)$ in \eqref{BD} coincides
with the action of $1/t$ on
$\mathbb{C}^* \times \Mrep(X_{\mathbb{R}})
 \, \longrightarrow\, \Mlambda(\overline{X})$. Therefore,
we get an action of $\mathbb{C}^*$ on $\MDH(X)$.
Let
\begin{equation}\label{eta}
\eta\, :\, \MDH(X)^{\smooth}\,\longrightarrow \,
T \MDH(X)^{\smooth}
\end{equation}
be the holomorphic vector field defined by this action
of $\mathbb{C}^*$.

Let $Z \,\subset\, \MDH(X)$ be a closed analytic
subset with the following three properties:
\begin{itemize}
 \item $Z$ is irreducible and has
  complex dimension $(r^2-1)(g-1)$.
 \item The smooth locus $Z^{\smooth} \,\subseteq\, Z$
  lies in the smooth locus $\MDH(X)^{\smooth} \subset \MDH(X)$.
 \item The restriction of the holomorphic tangent
  bundle $T \MDH(X)^{\smooth}$ to the subspace
  $Z^{\smooth}$ has no nonzero holomorphic sections.
\end{itemize}
We noted above that both $\iota( \Mbun)$ and
$\iota( \Mbunbar)$ (see \eqref{subvar1} and
\eqref{subvar2}) satisfy these conditions.

The third condition
on $Z$ implies that the vector field $\eta$ in \eqref{eta}
vanishes on $Z^{\smooth}$. It follows that the fixed point locus
$\MDH(X)^{\mathbb{C}^*}$ contains $Z^{\smooth}$, and hence
also contains its closure $Z$. Therefore, using
Proposition \ref{fix_lambda} we conclude that $Z$ is
one of $\iota( \Mbun)$ and $\iota( \Mbunbar)$.
Using \cite[page 229, Theorem E]{KP} we now know that
the isomorphism class of the analytic space
$\MDH( X)$ determines the isomorphism class of the
unordered pair of
Riemann surfaces $\{X\, ,\overline{X}\}$. This completes
the proof of the theorem.
\end{proof}

\section*{Acknowledgements}
The first and second authors were supported by
the grant MTM2007-63582 of the Spanish
Ministerio de Educaci\'on y Ciencia.
The second author was also supported by
the grant 200650M066 of
Comunidad Aut\'onoma de Madrid.
The third author was supported by the
SFB/TR 45 `Periods, moduli spaces and
arithmetic of algebraic varieties'. The fourth autor was
supported by the grant SFRH/BPD/27039/2006 of the Funda\c{c}\~{a}o
para a Ci\^{e}ncia e a Tecnologia and CMUP, financed by F.C.T. (Portugal) through the programmes POCTI and POSI, with national and European Community structural funds. 


\end{document}